\documentclass[12pt]{amsart} 
\usepackage{amssymb}
\usepackage{xcolor}
\usepackage{fullpage}
\usepackage[colorlinks=true]{hyperref}
\usepackage{listings}
\usepackage[normalem]{ulem}
\usepackage{minted}
\usepackage{comment}

\newcommand{\bbQ}{\mathbb{Q}}
\newcommand{\bbR}{\mathbb{R}}
\newcommand{\bbZ}{\mathbb{Z}}

\newcommand{\cO}{\mathcal{O}}
\newcommand{\cT}{\mathcal{T}}

\newcommand{\legendre}[2]{\genfrac{(}{)}{}{}{#1}{#2}}

\font\cyr=wncyr10
\newcommand{\Sha}{\hbox{\cyr X}}

\DeclareMathOperator{\ord}{ord}

\DeclareMathOperator{\rank}{rank}

\DeclareMathOperator{\sgn}{sign}

\numberwithin{equation}{section}
\newtheorem{thm}{Theorem}[section]
\newtheorem{rem}{Remark}[section]
\newtheorem{lem}{Lemma}[section]
\newtheorem{theorem}{Theorem}[section]
\newtheorem{conj}[theorem]{Conjecture}

\newtheorem{proposition}[theorem]{Proposition}

\theoremstyle{definition}

\newtheorem*{remark-nonum}{Remark}
\newtheorem*{remarks-nonum}{Remarks}

\begin{document}

\title{Classification of the rank of a certain family of elliptic curves}

\author{ Arkabrata Ghosh, Bidisha Roy, and Richa Sharma }

\address[Arkabrata Ghosh]{ Department of Mathematics, SRM University-AP, Amaravati 522240, Andhra Pradesh, India}
\email[arkabrata Ghosh]{arka2686@gmail.com}

\address[Bidisha Roy]{Department of Mathematics and Statistics, Indian Institute of Technology Tirupati, Yerpedu, Andhra Pradesh 517619, India}
 \email[Bidisha Roy]{bidisha.roy@iittp.ac.in, brroy123456@gmail.com}

 \address[Richa Sharma]{Department of Mathematics, Jamia Milia Islamia University, New Delhi}
\email[Richa Sharma]{richasharma582@gmail.com}




\maketitle


\section*{Abstract}
In this article, we study the family of elliptic curves
$E_{-2pq}: y^2=x^3-2pqx$, where $p$ and $q$ are distinct odd primes.
Using a $2$-isogeny methods and some elementary techniques, we obtain explicit possibilities for the Mordell--Weil ranks, conditional on the Parity Conjecture. Moreover, in the rank-one case, we are also able to derive explicit conditions that are independent of the parity conjecture. Moreover, the main results depend only on the residue classes of $(p,q)$ modulo $8$ and the Legendre symbols $\legendre{p}{q}$.  

\section{Introduction}

The study of elliptic curves plays an important role in modern number theory and arithmetic geometry. One of the most fundamental properties to an elliptic curve $E/\mathbb{Q}$ is the rank of its Mordell-Weil group $E(\mathbb{Q})$. Using the Mordell-Weil Theorem, we know the group of rational points of an elliptic curve is finitely generated and hence admits the following decomposition:

$$
E(\mathbb{Q}) \cong E(\mathbb{Q})_{\mathrm{tors}} \oplus \mathbb{Z}^{r},
$$

where $E(\mathbb{Q})_{\mathrm{tors}}$ denotes the torsion subgroup and $r$ is the Mordell--Weil rank. Even today, determining the rank of elliptic curves remains one of the most challenging problems in algebraic Number Theory and is closely related to the Birch and Swinnerton-Dyer conjecture.

A particularly fruitful approach to studying the rank of elliptic curves that possess a rational point of order two is the method of $2$-descent. This method allows one to obtain information about the Mordell-Weil group through the analysis of certain associated homogeneous spaces (or torsors). Over the years, $2$-descent has been successfully applied to several families of elliptic curves of the form $y^2=x^3-Dx$, leading to explicit rank computations and descriptions of their arithmetic properties. Spearman \cite{Sp07,Sp07_1} investigated the curves $ y^2=x^3-px$ and $y^2=x^3-2px$, and obtained conditions on the prime $p$ under which these curves have rank at least two. Fujita and Terai \cite{FT11} considered the family
$
y^2=x^3-p^k x,
$
where $p$ is prime and $k=1,2,3$, and established necessary and sufficient conditions for the rank to be equal to one or two. More recently, in a couple of articles \cite{Gh25, Gh26}, Ghosh  investigated the family
$
y^2=x^3-5pqx,
$
and
$
y^2=x^3-pqx
$
and determined several properties of its torsion subgroup and rank. The techniques developed there relied heavily on the method of $2$-descent and the analysis of the associated homogeneous spaces.

Motivated by these results, in this paper, we consider the family
\begin{equation}
\label{eq:1.1}
    E_{-2pq}:~ y^2 = x^3 -2pqx,
\end{equation}where $p$ and $q$ are distinct odd primes. Our aim is to obtain
explicit conditional possibilities and upper bounds for the rank of this family.
Now, assuming the Parity Conjecture \ref{conj:parity}, we will prove the following Theorems.

\begin{thm}
    \label{thm:1.1} Suppose $p$ and $q$ are distinct odd primes such that $(p,q)= (1,1) \pmod 8 $.  Then 
  $$
\operatorname{rank}(E_{-2pq}(\mathbb{Q}))=
\begin{cases}
1 \ \text{or}\ 3 \ \text{or}\ 5, & \text{if } \legendre{p}{q}=1,~~ \\
1 \ \text{or}\ 3, & \text{if }  \legendre{p}{q}=-1.
\end{cases}
$$
\end{thm}

\begin{thm}
    \label{thm:1.2}

 Suppose $p$ and $q$ are distinct odd primes such that $(p,q) \equiv (1,3)$, $(3,1)$, $(1,5)$,  $(5,1)$, $(1,7)$ or $(7,1)\pmod 8.$ Then   we have
$$
\operatorname{rank}(E_{-2pq}(\mathbb{Q}))=
\begin{cases}
1 \ \text{or}\ 3, & \text{if } \legendre{p}{q}=1,~~ \\
1, & \text{if }  \legendre{p}{q}=-1.
\end{cases}$$

\end{thm}

\begin{thm}
    \label{thm:1.3} 
 Suppose $p$ and $q$ are distinct odd primes such that $(p,q)$ $\equiv  (3,3)$, $(3,5)$, $(5,3)$, $(3,7)$, $(7,3)$, $(5,7)$, $(7,5)$ or $(5,5) \pmod 8 $.  Then we have
  $$
\operatorname{rank}(E_{-2pq}(\mathbb{Q}))= 1
$$

for both $\legendre{p}{q}=1$ and $\legendre{p}{q}=-1$.

\end{thm}

\begin{thm}
    \label{thm:1.4} 
    
 \noindent  (i) Suppose $p$ and $q$ are distinct odd primes such that $(p,q)= (7,7) \pmod 8 $.   Then 
  $$
\operatorname{rank}(E_{-2pq}(\mathbb{Q}))= 1 \ \text{or}\ 3
$$

for both $\legendre{p}{q}=1$ and $\legendre{p}{q}=-1$.
\end{thm}

The following table verifies the hypotheses of Theorems \ref{thm:1.1}-\ref{thm:1.4}.

\begin{table}[h]
\centering
\scriptsize
\setlength{\tabcolsep}{2pt}
\begin{tabular*}{\textwidth}{@{\extracolsep{\fill}}ccc ccc ccc ccc ccc ccc@{}}
\hline
\multicolumn{9}{c}{Rank $1$ examples} &
\multicolumn{9}{c}{Rank $3$ examples} \\
\hline
$p$ & $q$ & Rank & $p$ & $q$ & Rank & $p$ & $q$ & Rank &
$p$ & $q$ & Rank & $p$ & $q$ & Rank & $p$ & $q$ & Rank \\
\hline
$3$  & $5$  & $1$ & $3$  & $7$  & $1$ & $3$  & $11$ & $1$ &
$3$  & $97$ & $3$ & $7$  & $23$ & $3$ & $13$ & $17$ & $3$ \\

$5$  & $7$  & $1$ & $5$  & $11$ & $1$ & $5$  & $13$ & $1$ &
$17$ & $43$ & $3$ & $17$ & $59$ & $3$ & $17$ & $73$ & $3$ \\

$7$  & $11$ & $1$ & $7$  & $13$ & $1$ & $7$  & $17$ & $1$ &
$17$ & $89$ & $3$ & $19$ & $73$ & $3$ & $23$ & $47$ & $3$ \\

$11$ & $13$ & $1$ & $11$ & $17$ & $1$ & $11$ & $19$ & $1$ &
$31$ & $71$ & $3$ & $41$ & $61$ & $3$ & $41$ & $73$ & $3$ \\

$13$ & $19$ & $1$ & $13$ & $23$ & $1$ & $13$ & $29$ & $1$ &
$47$ & $79$ & $3$ & $53$ & $97$ & $3$ & $61$ & $73$ & $3$ \\

$17$ & $19$ & $1$ & $17$ & $23$ & $1$ & $17$ & $29$ & $1$ &
$73$ & $97$ & $3$ & $79$ & $97$ & $3$ & $89$ & $97$ & $3$ \\
\hline
\end{tabular*}
\caption{Examples of ranks $1$ and $3$ for $E_{-2pq}: y^2=x^3-2pqx$, where $3\leq p<q\leq 97$.}
\label{tab:rank_one_three_examples}
\end{table}
In the next theorem, we provide analogous results of Theorem \ref{thm:1.3} without assuming the parity conjecture.

\begin{thm}\label{thm:1.5}
\begin{enumerate}
    \item[(i)]  Let $(p,q) \equiv (3,3) \pmod 8 $ such that $(pq-32)$ is a perfect square. Then $\operatorname{rank}(E_{-2pq}(\bbQ))=1$ for $\legendre{p}{q}=1$ and $\legendre{p}{q}=-1$.
    \item[(ii)] If $(p,q)\equiv(3,5)\pmod 8$, then
$\operatorname{rank}(E_{-2pq}(\bbQ))=1$ whenever
\[
\begin{array}{ll}
\legendre{p}{q}=1 & \text{and } p-2q \text{ is a perfect square or},\\
\legendre{p}{q}=-1 & \text{and } 2p-q \text{ is a perfect square}.
\end{array} 
\] holds.
\item[(iii)] If $(p,q)\equiv(5,3)\pmod 8$, then
$\operatorname{rank}(E_{-2pq}(\bbQ))=1$ whenever
\[
\begin{array}{ll}
\legendre{p}{q}=1
& \text{and } q-2p \text{ is a perfect square, or}\\
\legendre{p}{q}=-1
& \text{and } 2q-p \text{ is a perfect square}.
\end{array}
\] holds.
\item[(iv)]
If $(p,q)\equiv(3,7)\pmod 8$, then
$\operatorname{rank}(E_{-2pq}(\bbQ))=1$ whenever
\[
\begin{array}{ll}
\legendre{p}{q}=1
& \text{and } 32p-q \text{ is a perfect square, or}\\
\legendre{p}{q}=-1
& \text{and } q-2p \text{ is a perfect square}.
\end{array}
\] holds.
\item[(v)]
If $(p,q)\equiv(7,3)\pmod 8$, then
$\operatorname{rank}(E_{-2pq}(\bbQ))=1$ whenever
\[
\begin{array}{ll}
\legendre{p}{q}=1
& \text{and } p-2q \text{ is a perfect square, or}\\
\legendre{p}{q}=-1
& \text{and } 32q-p \text{ is a perfect square}.
\end{array}
\] holds.
\item[(vi)]
If $(p,q)\equiv(5,7)\pmod 8$, then
$\operatorname{rank}(E_{-2pq}(\bbQ))=1$ whenever
\[
\begin{array}{ll}
\legendre{p}{q}=1
& \text{and } 32p-q \text{ is a perfect square, or}\\
\legendre{p}{q}=-1
& \text{and } 2q-p \text{ is a perfect square}.
\end{array}
\] holds.
\item[(vii)] 
If $(p,q)\equiv(7,5)\pmod 8$, then
$\operatorname{rank}(E_{-2pq}(\bbQ))=1$ whenever
\[
\begin{array}{ll}
\legendre{p}{q}=1
& \text{and } 2p-q \text{ is a perfect square, or}\\
\legendre{p}{q}=-1
& \text{and } 32q-p \text{ is a perfect square}.
\end{array}
\] holds.
\item[(viii)]
If $(p,q)\equiv(5,5)\pmod 8$ and $2pq-1$ is a perfect square, then
\[
\operatorname{rank}(E_{-2pq}(\bbQ))=1
\]
for both $\legendre{p}{q}=1$ and $\legendre{p}{q}=-1$.
\end{enumerate}
\end{thm} 

\section{Preliminaries}

To compute the rank of the curve given by \eqref{eq:1.1}, we will use the method
of $2$-descent. We will describe it briefly here. See \cite[Chapter~III, Sections~4--6]{ST15}
for more details. Suppose that $E: Y^{2}= X^{3} + aX^{2} + bX$ is an elliptic curve
over $\bbQ$ and $\overline{E}: Y^{2} = X^{3} - 2aX^{2} + \left( a^{2}- 4b \right)X$
is the corresponding isogenous curve to $E$. Hence, there exists an isogeny
$\phi: E \rightarrow \overline{E}$ of degree $2$ given by
\[
\phi(x,y) = \left( \frac{y^{2}}{x^{2}}, \frac{y(x^{2}-b)}{x^{2}} \right).
\]

Moreover, let $\bbQ^{*}$ be the multiplicative group of all non-zero rational numbers,
and ${\bbQ^{*}}^{2}$ be its subgroup of squares of elements of $\bbQ^{*}$. Hence,
$\bbQ^{*}/{\bbQ^{*}}^{2}$ is the multiplicative group of all non-zero rational
numbers modulo squares. Now we define the $2$-descent
homomorphism $\alpha: E(\bbQ) \rightarrow \bbQ^{*}/{\bbQ^{*}}^{2}$ by
\[
\alpha(P)= 
\begin{cases}
1 \pmod{{\bbQ^{*}}^{2}}, ~\text{if}~ P= \cO, ~\text{the point at infinity},\\
b \pmod{{\bbQ^{*}}^{2}}, ~\text{if}~ P= (0,0),\\
x \pmod{{\bbQ^{*}}^{2}}, ~\text{if}~ P= (x,y) ~\text{with}~ x \neq 0.
\end{cases}
\]

Similarly, we can define the $2$-decent homomorphism on the isogenous curve
$\overline{E}(\bbQ)$ as follows:
$\overline{\alpha}: \overline{E}(\bbQ) \rightarrow \bbQ^{*}/{\bbQ^{*}}^{2}$ by
\[
\overline{\alpha}(\overline{P})
=\begin{cases}
1 \pmod {{\bbQ^{*}}^{2}}, ~\text{if}~ \overline{P}= \overline{\cO}, ~\text{the point at infinity},\\
\overline{b} \pmod {{\bbQ^{*}}^{2}}, ~\text{if}~ \overline{P}= (0,0),\\
x \pmod {{\bbQ^{*}}^{2}}, ~\text{if}~ \overline{P}= (x,y) ~\text{with}~ x \neq 0,
\end{cases}
\]
where $\overline{b}= a^{2} - 4b$.

The group $\alpha(E(\bbQ))$ consists of $1,b$ and all factors $b_{1} \neq 1,b$
of $b$ taken modulo ${\bbQ^{*}}^{2}$ that satisfy the following conditions.
For $b_{1} \neq 1, ~\text{or}~ b \pmod {{\bbQ^{*}}^{2}}$,
there exists a triple $(N, M, e) \in \bbZ^{3}$, where $M \neq 0, e \neq 0$ is a
solution of the Diophantine equation (or 'torsor')
\[
\cT: N^{2} = b_{1}M^{4} + aM^{2}e^{2} + b_{2}e^{4}, ~\text{with}~  b_{1} b_{2} =b,
\]
and satisfies the following criterion:
\begin{equation}
\label{eq:gcd}
\gcd(N,e)=\gcd(M,e)=\gcd \left( b_{1}, e \right)=\gcd \left( b_{2}, M \right)
= \gcd(M,N)=1.
\end{equation}

Similarly, the group $\overline{\alpha} \left( \overline{E}(\bbQ) \right)$
consists of $1, a^{2}-4b$ and all factors $b_{1}$ of $a^{2}-4b$ taken modulo
${\bbQ^{*}}^{2}$ that satisfy the following conditions.
For $b_{1} \neq 1, ~\text{or}~ a^{2}-4b \pmod{{\bbQ^{*}}^{2}}$,
there exists a triple $(N, M, e) \in \bbZ^{3}$,
where $M \neq 0, e \neq 0$ is a solution of the Diophantine equation (or 'torsor'),
\[
\overline{\cT}: N^{2} = b_{1} M^{4} -2a M^{2} e^{2} + b_{2} e^{4}, ~\text{with}~  b_{1} b_{2} =a^{2}-4b,
\]
and the same gcd criterion as in \eqref{eq:gcd} above holds.

Now, to compute the rank of an elliptic curve, $E$, of the above form, we use the
following proposition (see the displayed equation at the bottom of page~91 of
\cite{ST15}).

\begin{proposition}
\label{prop:2.1}
Let $r$ be the rank of $E(\bbQ)$, where $E$, $\alpha$ and $\overline{\alpha}$
are as above. Then,
\[
\frac{1}{4} |\alpha(E(\bbQ))| \cdot \left| \overline{\alpha} \left( \overline{E}(\bbQ) \right) \right|
= 2^{r}.
\]
\end{proposition}
Next, we recall another concept called {\it root number}  for an elliptic curve, $E$, defined over $\bbQ$. It is well-known that there is an $L$-function associated with it, $L(E/\bbQ,s)$. It satisfies the functional equation
\[
L^{*}(E/\bbQ,2-s)=w(E/\bbQ)L^{*}(E/\bbQ,s),
\]
where $w(E/\bbQ)=\pm 1$ is called the \emph{global root number}. The global root number $w(E/\bbQ)$ is defined as the product of \emph{local root numbers}
$w \left( E/\bbQ_{v} \right) \in \{ \pm 1 \}$,
\[
w(E/\bbQ) = \prod_{v} w \left( E/\bbQ_{v} \right),
\]
where the product runs over all places $v$ of $\bbQ$, including the infinite one.
For more information about root numbers, see \cite{CK-D}. We list the following lemmas, which will be useful to prove the main results.

\begin{lem}
\label{lem:root}
Let $p$ and $q$ be distinct odd primes. Then
\[
w\left(E_{-2pq}/\bbQ\right)=-1.
\]
\end{lem}

\begin{proof}
Put $D=2pq$, so that $E_{-2pq}$ is the curve $E_D:y^2=x^3-Dx$.
By the formula of Birch and Stephens \cite[(9) and (10), p.~296]{B-S},
\[
w(E_D/\bbQ)=w(E_D/\bbR)w(E_D/\bbQ_2)
\prod_{\ell^2\parallel D}w(E_D/\bbQ_\ell).
\]
Since $D>0$, we have $w(E_D/\bbR)=\sgn(-D)=-1$. Moreover,
$D=2pq$ is even, and hence it is not congruent to
$1,3,11,$ or $13\pmod{16}$; therefore $w(E_D/\bbQ_2)=1$.
Finally, $D$ is square-free, so there is no odd prime $\ell$ satisfying
$\ell^2\parallel D$, and the product is empty. Consequently,
$w(E_{-2pq}/\bbQ)=-1$.
\end{proof}

\begin{conj}[Parity Conjecture]
\label{conj:parity}
Let $E/\bbQ$ be an elliptic curve defined over $\bbQ$. Then
\[
(-1)^{\rank(E(\bbQ))}=w(E/\bbQ).
\]
\end{conj}

\begin{lem}
\label{lem:parity}(\cite{D-D},  Theorem~1.2)
Let $E/\bbQ$ be an elliptic curve defined over $\bbQ$. If the Tate-Shafarevich
group, $\Sha$ $(E/\bbQ)$, is finite, then the Parity Conjecture holds.
\end{lem}

\begin{lem}
\label{lem:kolyvagin}
Let $E/\bbQ$ be an elliptic curve defined over $\bbQ$. If
$\ord_{s=1}L(E/\bbQ,s) \leq 1$, then $\rank(E(\bbQ))=\ord_{s=1}L(E/\bbQ,s)$ and  $\Sha$ $(E/\bbQ)$, is finite.
\end{lem}



\section{\texorpdfstring{Rank of $E_{-2pq}$}{Rank of E_{-2pq}}}

\subsection{}

Throughout this section, we assume that $p\neq q$ are distinct odd primes. To
compute the rank of $E_{-2pq}$, we first determine
$\alpha(E_{-2pq}(\bbQ))$. By definition,
$1,-2pq\in\alpha(E_{-2pq}(\bbQ))$.

Since $\alpha(E_{-2pq}(\bbQ))$ is a subgroup of
$\bbQ^*/{\bbQ^*}^2$ containing $-2pq$, the two square classes $d$ and
$-2pq/d$ occur simultaneously. Moreover, the torsors corresponding to the
factorizations $-2pq=b_1b_2$ and $-2pq=b_2b_1$ are obtained from one another
by interchanging $M$ and $e$. Hence it is enough to choose one representative
from each complementary pair. We take
\[
T=\{-1,2,-2,q,-q,p,-p\}.
\]
The complementary classes are, respectively,
$2pq,-pq,pq,-2p,2p,-2q,$ and $2q$.

We therefore need to consider only the following seven torsors over the integers.

\begin{table}[h]
\centering
\[
\begin{array}{c@{\qquad}c@{\qquad}l}
b_1 & b_2 & \text{Torsor} \\
\hline
\noalign{\vskip 3pt}
-1 & 2pq  & \cT_1:\; N^2=-M^4+2pqe^4 \\
 2 & -pq  & \cT_2:\; N^2=2M^4-pqe^4 \\
-2 & pq   & \cT_3:\; N^2=-2M^4+pqe^4 \\
 q & -2p  & \cT_4:\; N^2=qM^4-2pe^4 \\
-q & 2p   & \cT_5:\; N^2=-qM^4+2pe^4 \\
 p & -2q  & \cT_6:\; N^2=pM^4-2qe^4 \\
-p & 2q   & \cT_7:\; N^2=-pM^4+2qe^4 \\
\hline
\end{array}
\]
\caption{Torsors for $E_{-2pq}$.}
\label{table:T-torsors}
\end{table}

\begin{lem}
\label{lem:T1}
There exist integer solutions for the torsor $\cT_{1}: N^{2}=-M^{4}+2pqe^{4}$ satisfying \eqref{eq:gcd} only if $p\equiv q\equiv 1 \pmod 4$ and $pq\equiv 1,~5 \pmod 8$.
\end{lem}

\begin{proof}
Reducing $\cT_{1}$ modulo $p$, we get $N^{2}\equiv -M^{4}\pmod p$. Since $p\nmid M$ by \eqref{eq:gcd}, we get $\left(\dfrac{-1}{p}\right)=1$, and hence $p\equiv 1\pmod 4$. Similarly, reducing $\cT_{1}$ modulo $q$, we get $\left(\dfrac{-1}{q}\right)=1$, and hence $q\equiv 1\pmod 4$.

Now reducing $\cT_{1}$ modulo $16$, from \eqref{eq:gcd}, we deduce that $M$ is odd and hence $N$ is odd. Thus $N^{2}\equiv 1,~9\pmod {16}$ and $M^{4}\equiv 1\pmod {16}$. If $e$ is even, then $e^{4}\equiv 0\pmod {16}$, and so $N^{2}\equiv -1\pmod {16}$, which is impossible. Hence $e$ is odd. Therefore $e^{4}\equiv 1\pmod {16}$, and using this in $N^{2}=-M^{4}+2pqe^{4}$, we get $2pq-1\equiv 1,~9\pmod {16}$. After simplification, we obtain $pq\equiv 1,~5\pmod 8$.
\end{proof}

\smallskip
\begin{lem}
\label{lem:T2}
There exist integer solutions for the torsor $\cT_{2}: N^{2}=2M^{4}-pqe^{4}$ satisfying \eqref{eq:gcd} only if $\left(\dfrac{2}{p}\right)=1$ and $\left(\dfrac{2}{q}\right)=1$. Equivalently, $p,q\equiv 1$ or $7\pmod 8$. Moreover, reducing modulo $16$, one obtains either $pq\equiv 1,~9\pmod {16}$ or $pq\equiv 7,~15\pmod {16}$.
\end{lem}

\begin{proof}
Reducing $\cT_{2}$ modulo $p$, we get $N^{2}\equiv 2M^{4}\pmod p$. Since $p\nmid M$ by \eqref{eq:gcd}, we get $\left(\dfrac{2}{p}\right)=1$, and hence $p\equiv 1$ or $7\pmod 8$. Similarly, reducing $\cT_{2}$ modulo $q$, we get $\left(\dfrac{2}{q}\right)=1$, and hence $q\equiv 1$ or $7\pmod 8$.

Now, from \eqref{eq:gcd}, we have $e$ odd. Hence $e^{4}\equiv 1\pmod {16}$. If $M$ is odd, then $M^{4}\equiv 1\pmod {16}$, and reducing $\cT_{2}$ modulo $16$, we get $2-pq\equiv 1,~9\pmod {16}$. Therefore $pq\equiv 1,~9\pmod {16}$. If $M$ is even, then $M^{4}\equiv 0\pmod {16}$, and so $-pq\equiv 1,~9\pmod {16}$. Therefore $pq\equiv 7,~15\pmod {16}$.
\end{proof}

\smallskip

\begin{lem}
\label{lem:T3}
There exist integer solutions for the torsor $\cT_{3}: N^{2}=-2M^{4}+pqe^{4}$ satisfying \eqref{eq:gcd} only if $\left(\dfrac{-2}{p}\right)=1$ and $\left(\dfrac{-2}{q}\right)=1$. Equivalently, $p,q\equiv 1$ or $3\pmod 8$. Moreover, reducing modulo $16$, one obtains either $pq\equiv 3,~11\pmod {16}$ or $pq\equiv 1,~9\pmod {16}$.
\end{lem}

\begin{proof}
Reducing $\cT_{3}$ modulo $p$, we get $N^{2}\equiv -2M^{4}\pmod p$. Since $p\nmid M$ by \eqref{eq:gcd}, we get $\left(\dfrac{-2}{p}\right)=1$, and hence $p\equiv 1$ or $3\pmod 8$. Similarly, reducing $\cT_{3}$ modulo $q$, we get $\left(\dfrac{-2}{q}\right)=1$, and hence $q\equiv 1$ or $3\pmod 8$.

Now, from \eqref{eq:gcd}, we have $e$ odd. Hence $e^{4}\equiv 1\pmod {16}$. If $M$ is odd, then $M^{4}\equiv 1\pmod {16}$, and reducing $\cT_{3}$ modulo $16$, we get $-2+pq\equiv 1,~9\pmod {16}$. Therefore $pq\equiv 3,~11\pmod {16}$. If $M$ is even, then $M^{4}\equiv 0\pmod {16}$, and so $pq\equiv 1,~9\pmod {16}$.
\end{proof}


\smallskip

\begin{lem}
\label{lem:T45}
There exist integer solutions for the torsors $\cT_{4}: N^{2}=qM^{4}-2pe^{4}$ and $\cT_{5}: N^{2}=-qM^{4}+2pe^{4}$ satisfying \eqref{eq:gcd} only under the following necessary conditions. For $\cT_{4}$, one must have $\left(\dfrac{q}{p}\right)=1$ and $\left(\dfrac{-2p}{q}\right)=1$, with either $q\equiv 1\pmod 8$ or $q-2p\equiv 1,~9\pmod {16}$. For $\cT_{5}$, one must have $\left(\dfrac{-q}{p}\right)=1$ and $\left(\dfrac{2p}{q}\right)=1$, with either $q\equiv 7\pmod 8$ or $2p-q\equiv 1,~9\pmod {16}$.
\end{lem}

\begin{proof}
First consider $\cT_{4}: N^{2}=qM^{4}-2pe^{4}$. Reducing $\cT_{4}$ modulo $p$, we get $N^{2}\equiv qM^{4}\pmod p$. Since $p\nmid M$ by \eqref{eq:gcd}, we get $\left(\dfrac{q}{p}\right)=1$. Reducing $\cT_{4}$ modulo $q$, we get $N^{2}\equiv -2pe^{4}\pmod q$. Since $q\nmid e$ by \eqref{eq:gcd}, we get $\left(\dfrac{-2p}{q}\right)=1$.

Now reducing $\cT_{4}$ modulo $16$, from \eqref{eq:gcd}, we deduce that $M$ is odd and hence $N$ is odd. Thus $N^{2}\equiv 1,~9\pmod {16}$ and $M^{4}\equiv 1\pmod {16}$. If $e$ is even, then $e^{4}\equiv 0\pmod {16}$, and so $N^{2}\equiv q\pmod {16}$. Therefore $q\equiv 1,~9\pmod {16}$, which implies $q\equiv 1\pmod 8$. If $e$ is odd, then $e^{4}\equiv 1\pmod {16}$, and using this in $N^{2}=qM^{4}-2pe^{4}$, we get $q-2p\equiv 1,~9\pmod {16}$.

Now consider $\cT_{5}: N^{2}=-qM^{4}+2pe^{4}$. Reducing $\cT_{5}$ modulo $p$, we get $N^{2}\equiv -qM^{4}\pmod p$. Since $p\nmid M$ by \eqref{eq:gcd}, we get $\left(\dfrac{-q}{p}\right)=1$. Reducing $\cT_{5}$ modulo $q$, we get $N^{2}\equiv 2pe^{4}\pmod q$. Since $q\nmid e$ by \eqref{eq:gcd}, we get $\left(\dfrac{2p}{q}\right)=1$.

Again, reducing $\cT_{5}$ modulo $16$, from \eqref{eq:gcd}, we get that $M$ is odd and hence $N$ is odd. If $e$ is even, then $e^{4}\equiv 0\pmod {16}$, and so $N^{2}\equiv -q\pmod {16}$. Therefore $-q\equiv 1,~9\pmod {16}$, which implies $q\equiv 7\pmod 8$. If $e$ is odd, then $e^{4}\equiv 1\pmod {16}$, and using this in $N^{2}=-qM^{4}+2pe^{4}$, we get $2p-q\equiv 1,~9\pmod {16}$.
\end{proof}

\smallskip

\begin{lem}
\label{lem:T67}
There exist integer solutions for the torsors $\cT_{6}: N^{2}=pM^{4}-2qe^{4}$ and $\cT_{7}: N^{2}=-pM^{4}+2qe^{4}$ satisfying \eqref{eq:gcd} only under the following necessary conditions. For $\cT_{6}$, one must have $\left(\dfrac{p}{q}\right)=1$ and $\left(\dfrac{-2q}{p}\right)=1$, with either $p\equiv 1\pmod 8$ or $p-2q\equiv 1,~9\pmod {16}$. For $\cT_{7}$, one must have $\left(\dfrac{-p}{q}\right)=1$ and $\left(\dfrac{2q}{p}\right)=1$, with either $p\equiv 7\pmod 8$ or $2q-p\equiv 1,~9\pmod {16}$.
\end{lem}

\begin{proof}

Interchanging $p$ and $q$ in Lemma~\ref{lem:T45} sends
$\cT_4$ to $\cT_6$ and $\cT_5$ to $\cT_7$. Under the same interchange,
the Legendre-symbol and congruence conditions in Lemma~\ref{lem:T45} become
exactly the conditions stated above. Thus the result follows directly from
Lemma~\ref{lem:T45}.

\end{proof}

\smallskip

\smallskip

\begin{rem}
\label{rem:alpha_case}
From Lemmas~\ref{lem:T1}--\ref{lem:T67}, we have seen that the possible extra
elements of $\alpha\left(E_{-2pq}(\bbQ)\right)$ can arise only from the torsors
$\cT_{1},\ldots,\cT_{7}$. Therefore, for each residue class of $(p,q)$ modulo $8$,
we check the solvability conditions of these torsors and obtain an upper bound for
$\alpha\left(E_{-2pq}(\bbQ)\right)$.
\end{rem}

So, using Lemmas~\ref{lem:T1}--\ref{lem:T67}, we obtain the following results.

\medskip

\subsection*{Case: I}

At first, we shall consider the case when $\left(\dfrac{2}{p}\right)=1=\left(\dfrac{2}{q}\right)$.
It gives $(p,q)\equiv (1,1),(1,7),(7,1),(7,7)\pmod 8$.

\begin{lem}
\label{lem:alpha_case_1.1}
If $(p,q)\equiv (1,1)\pmod 8$, then
\[
\alpha\left(E_{-2pq}(\bbQ)\right)\subseteq
\begin{cases}
\{1,-1,2,-2,p,-p,q,-q,2p,-2p,2q,-2q,pq,-pq,2pq,-2pq\},
& \text{for } \legendre{p}{q}=1,\\
\{1,-1,2,-2,pq,-pq,2pq,-2pq\},
& \text{for } \legendre{p}{q}=-1.
\end{cases}
\]
\end{lem}
\begin{proof}
$\cT_{1}$: Since $p\equiv q\equiv 1\pmod 8$, Lemma~\ref{lem:T1} shows that $\cT_{1}$ may have integer solutions. Hence $-1$ and $2pq$ may be elements of $\alpha\left(E_{-2pq}(\bbQ)\right)$.

$\cT_{2}$ and $\cT_{3}$: Since $p\equiv q\equiv 1\pmod 8$, the congruence conditions in Lemmas~\ref{lem:T2} and \ref{lem:T3} are satisfied. Hence $2,-pq,-2,pq$ may be elements of $\alpha\left(E_{-2pq}(\bbQ)\right)$.

$\cT_{4}$ and $\cT_{5}$: Here $\left(\dfrac{q}{p}\right)=\left(\dfrac{p}{q}\right)$ by quadratic reciprocity. If $\left(\dfrac{p}{q}\right)=1$, then all the Legendre-symbol conditions in Lemma~\ref{lem:T45} are satisfied, and hence $q,-2p,-q,2p$ may be elements of $\alpha\left(E_{-2pq}(\bbQ)\right)$. If $\left(\dfrac{p}{q}\right)=-1$, then the Legendre-symbol conditions fail, and hence these torsors have no integer solutions.

$\cT_{6}$ and $\cT_{7}$: Similarly, if $\left(\dfrac{p}{q}\right)=1$, then all the Legendre-symbol conditions in Lemma~\ref{lem:T67} are satisfied, and hence $p,-2q,-p,2q$ may be elements of $\alpha\left(E_{-2pq}(\bbQ)\right)$. If $\left(\dfrac{p}{q}\right)=-1$, then these torsors have no integer solutions.

Combining the above possibilities, we get the desired result.
\end{proof}

\begin{lem}
\label{lem:alpha_case_1.2}
If $(p,q)\equiv (1,7)\pmod 8$, then
\[
\alpha\left(E_{-2pq}(\bbQ)\right)\subseteq
\begin{cases}
\{1,2,p,-q,2p,-2q,-pq,-2pq\},
& \text{for } \legendre{p}{q}=1,\\
\{1,2,-pq,-2pq\},
& \text{for } \legendre{p}{q}=-1.
\end{cases}
\]
\end{lem}


\begin{lem}
\label{lem:alpha_case_1.3}
If $(p,q)\equiv (7,1)\pmod 8$, then
\[
\alpha\left(E_{-2pq}(\bbQ)\right)\subseteq
\begin{cases}
\{1,2,-p,q,-2p,2q,-pq,-2pq\},
& \text{for } \legendre{p}{q}=1,\\
\{1,2,-pq,-2pq\},
& \text{for } \legendre{p}{q}=-1.
\end{cases}
\]
\end{lem}

\begin{proof}
This is obtained from Lemma~\ref{lem:alpha_case_1.2} by interchanging $p$ and $q$.
\end{proof}

\begin{lem}
\label{lem:alpha_case_1.4}
If $(p,q)\equiv (7,7)\pmod 8$, then
\[
\alpha\left(E_{-2pq}(\bbQ)\right)\subseteq
\begin{cases}
\{1,2,p,-q,2p,-2q,-pq,-2pq\},
& \text{for } \legendre{p}{q}=1,\\
\{1,2,-p,q,-2p,2q,-pq,-2pq\},
& \text{for } \legendre{p}{q}=-1.
\end{cases}
\]
\end{lem}


\subsection*{Case: II}Now we shall consider the case when $\left(\dfrac{2}{p}\right)=1$ and
$\left(\dfrac{2}{q}\right)=-1$. It gives $(p,q)\equiv (1,3),(1,5),(7,3),(7,5)\pmod 8$.

\begin{lem}
\label{lem:alpha_case_2.1}
If $(p,q)\equiv (1,3)\pmod 8$, then
\[
\alpha\left(E_{-2pq}(\bbQ)\right)\subseteq
\begin{cases}
\{1,-2,p,q,-2p,-2q,pq,-2pq\},
& \text{for } \legendre{p}{q}=1,\\
\{1,-2,pq,-2pq\},
& \text{for } \legendre{p}{q}=-1.
\end{cases}
\]
\end{lem}

\begin{proof}
$\cT_{1}$: Since $q\equiv 3\pmod 8$, Lemma~\ref{lem:T1} shows that $\cT_{1}$ has no integer solutions.

$\cT_{2}$: Since $q\equiv 3\pmod 8$, Lemma~\ref{lem:T2} shows that $\cT_{2}$ has no integer solutions.

$\cT_{3}$: Since $p\equiv 1\pmod 8$ and $q\equiv 3\pmod 8$, Lemma~\ref{lem:T3} shows that $\cT_{3}$ may have integer solutions. Hence $-2$ and $pq$ may be elements of $\alpha\left(E_{-2pq}(\bbQ)\right)$.

$\cT_{4}$: Since $p\equiv 1\pmod 8$, by quadratic reciprocity we have
$
\left(\dfrac{q}{p}\right)=\left(\dfrac{p}{q}\right).
$
Moreover, since $q\equiv 3\pmod 8$, we have $\left(\dfrac{-2}{q}\right)=1$. Hence
$
\left(\dfrac{-2p}{q}\right)=\left(\dfrac{p}{q}\right).
$
Therefore, if $\left(\dfrac{p}{q}\right)=1$, then the Legendre-symbol conditions of $\cT_{4}$ are satisfied. Also $q-2p\equiv 1\pmod 8$, so the congruence condition in Lemma~\ref{lem:T45} may hold. Hence $q$ and $-2p$ may be elements of $\alpha\left(E_{-2pq}(\bbQ)\right)$. If $\left(\dfrac{p}{q}\right)=-1$, then the Legendre-symbol conditions fail, so $\cT_{4}$ has no integer solutions.

$\cT_{5}$: Since $q\equiv 3\pmod 8$, the congruence condition in Lemma~\ref{lem:T45} for $\cT_{5}$ is not satisfied. Hence $\cT_{5}$ has no integer solutions.

$\cT_{6}$: If $\left(\dfrac{p}{q}\right)=1$, then the Legendre-symbol conditions in Lemma~\ref{lem:T67} are satisfied, and since $p\equiv 1\pmod 8$, the congruence condition also holds. Hence $p$ and $-2q$ may be elements of $\alpha\left(E_{-2pq}(\bbQ)\right)$. If $\left(\dfrac{p}{q}\right)=-1$, then $\cT_{6}$ has no integer solutions.

$\cT_{7}$: Since $p\equiv 1\pmod 8$, the congruence condition in Lemma~\ref{lem:T67} for $\cT_{7}$ is not satisfied. Hence $\cT_{7}$ has no integer solutions.

Combining the above possibilities, we get the desired result.
\end{proof}

\begin{lem}
\label{lem:alpha_case_2.2}
If $(p,q)\equiv (1,5)\pmod 8$, then
\[
\alpha\left(E_{-2pq}(\bbQ)\right)\subseteq
\begin{cases}
\{1,-1,p,-p,2q,-2q,2pq,-2pq\},
& \text{for } \legendre{p}{q}=1,\\
\{1,-1,2pq,-2pq\},
& \text{for } \legendre{p}{q}=-1.
\end{cases}
\]
\end{lem}

\begin{lem}
\label{lem:alpha_case_2.3}
If $(p,q)\equiv (7,3)\pmod 8$, then
\[
\alpha\left(E_{-2pq}(\bbQ)\right)\subseteq
\begin{cases}
\{1,p,-2q,-2pq\},
& \text{for } \legendre{p}{q}=1,\\
\{1,-p,2q,-2pq\},
& \text{for } \legendre{p}{q}=-1.
\end{cases}
\]
\end{lem}


\begin{lem}
\label{lem:alpha_case_2.4}
If $(p,q)\equiv (7,5)\pmod 8$, then
\[
\alpha\left(E_{-2pq}(\bbQ)\right)\subseteq
\begin{cases}
\{1,-p,2q,-2pq\},
& \text{for } \legendre{p}{q}=1,\\
\{1,-q,2p,-2pq\},
& \text{for } \legendre{p}{q}=-1.
\end{cases}
\]
\end{lem}


\subsection*{Case: III }

Now we consider the case when $\left(\dfrac{2}{p}\right)=-1$ and
$\left(\dfrac{2}{q}\right)=1$, so that
$(p,q)\equiv(3,1),(3,7),(5,1),(5,7)\pmod 8$.
The four statements below are obtained, respectively, from
Lemmas~\ref{lem:alpha_case_2.1}, \ref{lem:alpha_case_2.3},
\ref{lem:alpha_case_2.2}, and \ref{lem:alpha_case_2.4}
by interchanging $p$ and $q$. For the class $(3,7)$, quadratic reciprocity
reverses the sign of the Legendre symbol; in the other three classes, it
preserves the sign.

\begin{lem}
\label{lem:alpha_case_3.1}
If $(p,q)\equiv (3,1)\pmod 8$, then
\[
\alpha\left(E_{-2pq}(\bbQ)\right)\subseteq
\begin{cases}
\{1,-2,p,q,-2p,-2q,pq,-2pq\},
& \text{for } \legendre{p}{q}=1,\\
\{1,-2,pq,-2pq\},
& \text{for } \legendre{p}{q}=-1.
\end{cases}
\]
\end{lem}

\begin{lem}
\label{lem:alpha_case_3.2}
If $(p,q)\equiv (3,7)\pmod 8$, then
\[
\alpha\left(E_{-2pq}(\bbQ)\right)\subseteq
\begin{cases}
\{1,-q,2p,-2pq\},
& \text{for } \legendre{p}{q}=1,\\
\{1,q,-2p,-2pq\},
& \text{for } \legendre{p}{q}=-1.
\end{cases}
\]
\end{lem}

\begin{lem}
\label{lem:alpha_case_3.3}
If $(p,q)\equiv (5,1)\pmod 8$, then
\[
\alpha\left(E_{-2pq}(\bbQ)\right)\subseteq
\begin{cases}
\{1,-1,q,-q,2p,-2p,2pq,-2pq\},
& \text{for } \legendre{p}{q}=1,\\
\{1,-1,2pq,-2pq\},
& \text{for } \legendre{p}{q}=-1.
\end{cases}
\]
\end{lem}

\begin{lem}
\label{lem:alpha_case_3.4}
If $(p,q)\equiv (5,7)\pmod 8$, then
\[
\alpha\left(E_{-2pq}(\bbQ)\right)\subseteq
\begin{cases}
\{1,-q,2p,-2pq\},
& \text{for } \legendre{p}{q}=1,\\
\{1,-p,2q,-2pq\},
& \text{for } \legendre{p}{q}=-1.
\end{cases}
\]
\end{lem}

\subsection*{Case: IV}

Finally, we shall consider the case when $\left(\dfrac{2}{p}\right)=-1=
\left(\dfrac{2}{q}\right)$. It gives 
$(p,q)\equiv (3,3),(3,5),(5,3),(5,5)\pmod 8$.

\begin{lem}
\label{lem:alpha_case_4.1}
If $(p,q)\equiv (3,3)\pmod 8$, then
\[
\alpha\left(E_{-2pq}(\bbQ)\right)\subseteq \{1,-2,pq,-2pq\}
\]
for both $\legendre{p}{q}=1$ and $\legendre{p}{q}=-1$.
\end{lem}

\begin{proof}
Here $\cT_{1}$ and $\cT_{2}$ have no integer solutions by Lemmas~\ref{lem:T1} and \ref{lem:T2}. Since $p\equiv q\equiv 3\pmod 8$, Lemma~\ref{lem:T3} shows that $\cT_{3}$ may have integer solutions. Hence $-2$ and $pq$ may be elements of $\alpha\left(E_{-2pq}(\bbQ)\right)$. Also, Lemmas~\ref{lem:T45} and \ref{lem:T67} show that $\cT_{4},\cT_{5},\cT_{6}$ and $\cT_{7}$ have no integer solutions. Therefore, we get the desired result.
\end{proof}

\begin{lem}
\label{lem:alpha_case_4.2}
If $(p,q)\equiv (3,5)\pmod 8$, then
\[
\alpha\left(E_{-2pq}(\bbQ)\right)\subseteq
\begin{cases}
\{1,p,-2q,-2pq\},
& \text{for } \legendre{p}{q}=1,\\
\{1,-q,2p,-2pq\},
& \text{for } \legendre{p}{q}=-1.
\end{cases}
\]
\end{lem}

\begin{proof}
From Lemmas~\ref{lem:T1}, \ref{lem:T2} and \ref{lem:T3}, the torsors $\cT_{1},\cT_{2}$ and $\cT_{3}$ have no integer solutions. Also, $\cT_{4}$ and $\cT_{7}$ have no integer solutions. If $\left(\dfrac{p}{q}\right)=1$, then $\cT_{6}$ may have integer solutions, and hence $p$ and $-2q$ may be elements of $\alpha\left(E_{-2pq}(\bbQ)\right)$. If $\left(\dfrac{p}{q}\right)=-1$, then $\cT_{5}$ may have integer solutions, and hence $-q$ and $2p$ may be elements of $\alpha\left(E_{-2pq}(\bbQ)\right)$. Combining these possibilities, we get the desired result.
\end{proof}

Interchanging $p$ and $q$, we get the following Lemma. 

\begin{lem}
\label{lem:alpha_case_4.3}
If $(p,q)\equiv (5,3)\pmod 8$, then
\[
\alpha\left(E_{-2pq}(\bbQ)\right)\subseteq
\begin{cases}
\{1,q,-2p,-2pq\},
& \text{for } \legendre{p}{q}=1,\\
\{1,-p,2q,-2pq\},
& \text{for } \legendre{p}{q}=-1.
\end{cases}
\]
\end{lem}


\begin{lem}
\label{lem:alpha_case_4.4}
If $(p,q)\equiv (5,5)\pmod 8$, then
\[
\alpha\left(E_{-2pq}(\bbQ)\right)\subseteq \{1,-1,2pq,-2pq\}
\]
for both $\legendre{p}{q}=1$ and $\legendre{p}{q}=-1$.
\end{lem}

\begin{proof}
Since $p\equiv q\equiv 5\pmod 8$, Lemma~\ref{lem:T1} shows that $\cT_{1}$ may have integer solutions. Hence $-1$ and $2pq$ may be elements of $\alpha\left(E_{-2pq}(\bbQ)\right)$. The remaining torsors $\cT_{2},\ldots,\cT_{7}$ have no integer solutions by Lemmas~\ref{lem:T2}, \ref{lem:T3}, \ref{lem:T45} and \ref{lem:T67}. Therefore, we get the desired result.
\end{proof}
\subsection{}

Now we determine $\left| \overline{\alpha} \left( \overline{E_{-2pq}}(\bbQ) \right) \right|$. As $\overline{E_{-2pq}}: y^{2}= x^{3} + 8pqx$, we know that $1, 2pq \in \overline{\alpha} \left( \overline{E_{-2pq}}(\bbQ) \right)$. Let us note the set of all possible divisors, $b_{1}$, of $8pq$ with $b_{1} \neq 1, 2pq$ modulo ${\bbQ^{*}}^{2}$ is

\[
\overline{T}=\{2,p,q,2p,2q,pq\}.
\]

We exclude all negative values of $b_{1}$ from the set $\overline{T}$
as the corresponding torsors $N^{2}= b_{1} M^{4} + b_{2} e^{4}$ will have no solutions
if both $b_{1}$ and $b_{2}$ are negative.

Since $\overline{\alpha}\left(\overline{E_{-2pq}}(\bbQ)\right)$ is a subgroup of
$\bbQ^{*}/{\bbQ^{*}}^{2}$ containing $2pq$, multiplication by $2pq$ gives the
complementary pairs
\[
\{2,pq\},\qquad \{p,2q\},\qquad \{q,2p\}.
\]
Thus, it is enough to study one representative from each pair. We choose the
representatives $p$, $q$, and $pq$, corresponding respectively to
$\overline{\cT_2}$, $\overline{\cT_3}$, and $\overline{\cT_6}$.


\begin{table}[h]
\centering
\[
\begin{array}{c@{\qquad}c@{\qquad}c@{\qquad}l}
\text{Label} & b_1 & b_2 & \text{Equation} \\
\hline
\noalign{\vskip 3pt}
\overline{\cT_1} & 2  & 4pq & N^2=2M^4+4pqe^4 \\
\overline{\cT_2} & p  & 8q  & N^2=pM^4+8qe^4 \\
\overline{\cT_3} & q  & 8p  & N^2=qM^4+8pe^4 \\
\overline{\cT_4} & 2p & 4q  & N^2=2pM^4+4qe^4 \\
\overline{\cT_5} & 2q & 4p  & N^2=2qM^4+4pe^4 \\
\overline{\cT_6} & pq & 8   & N^2=pqM^4+8e^4 \\
\hline
\end{array}
\]
\caption{Torsors for $\overline{E_{-2pq}}$.}
\label{table:oT-torsors}
\end{table}

\begin{lem}
\label{lem:oT1}
The classes represented by $\overline{\cT_1}$ and $\overline{\cT_6}$ occur
simultaneously in $\overline{\alpha}\left(\overline{E_{-2pq}}(\bbQ)\right)$.
Consequently, it is enough to study $\overline{\cT_6}$.
\end{lem}

\begin{proof}
Since $2pq$ belongs to
$\overline{\alpha}\left(\overline{E_{-2pq}}(\bbQ)\right)$, we have
\[
2\in\overline{\alpha}\left(\overline{E_{-2pq}}(\bbQ)\right)
\iff
2(2pq)\equiv pq\pmod{{\bbQ^{*}}^{2}}
\in\overline{\alpha}\left(\overline{E_{-2pq}}(\bbQ)\right).
\]
The classes $2$ and $pq$ correspond to $\overline{\cT_1}$ and
$\overline{\cT_6}$, respectively.
\end{proof}

\begin{lem}
    \label{lem:oT2}
There exist integer solutions for the torsor
$\overline{\cT_{2}}: N^{2}=pM^{4}+8qe^{4}$ only if
$\left(\dfrac{p}{q}\right)=1$, $\left(\dfrac{2q}{p}\right)=1$, and
$p\equiv 1\pmod 8$.
\end{lem}

\begin{proof}
Reducing $\overline{\cT_2}$ modulo $q$ and $p$, respectively, gives
$
\left(\dfrac{p}{q}\right)=1$ and $\left(\dfrac{2q}{p}\right)=1.$
By \eqref{eq:gcd}, $M$ is odd. Hence $N$ is also odd, and reducing the torsor
modulo $8$ gives
\[
1\equiv N^{2}\equiv pM^{4}\equiv p\pmod 8.
\]\end{proof}

\begin{lem}
    \label{lem:oT3}
There exist integer solutions for the torsor
$\overline{\cT_3}: N^2=qM^4+8pe^4$ only if
$\left(\dfrac{q}{p}\right)=1$, $\left(\dfrac{2p}{q}\right)=1$, and
$q\equiv1\pmod 8$.
\end{lem}

\begin{proof}
This follows from Lemma~\ref{lem:oT2} by interchanging $p$ and $q$.
\end{proof}

\begin{lem}
 \label{lem:oT4}
The classes represented by $\overline{\cT_4}$ and $\overline{\cT_3}$ occur
simultaneously, and the classes represented by $\overline{\cT_5}$ and
$\overline{\cT_2}$ occur simultaneously, in
$\overline{\alpha}\left(\overline{E_{-2pq}}(\bbQ)\right)$.
Consequently, no separate analysis of $\overline{\cT_4}$ and
$\overline{\cT_5}$ is required.
\end{lem}

\begin{proof}
Multiplication by $2pq$ gives
\[
(2p)(2pq)\equiv q
\qquad\text{and}\qquad
(2q)(2pq)\equiv p
\pmod{{\bbQ^{*}}^{2}}.
\]
Thus the pair $\{2p,q\}$ corresponds to
$\{\overline{\cT_4},\overline{\cT_3}\}$, while the pair $\{2q,p\}$
corresponds to $\{\overline{\cT_5},\overline{\cT_2}\}$.
\end{proof}

\begin{lem}
    \label{lem:oT6}
There exist integer solutions for the torsor $\overline{\cT_{6}}: N^{2}=pqM^{4}+8e^{4}$ only if $\left( \dfrac{2}{p} \right)=1$, $\left(\dfrac{2}{q} \right) =1$, and $ pq \equiv 1 \pmod 8$.  
\end{lem}
\begin{proof}
Reducing $\overline{\cT_6}$ modulo $p$ and $q$, respectively, gives $
\left(\dfrac{2}{p}\right)=1$ and 
$
\left(\dfrac{2}{q}\right)=1.
$
By \eqref{eq:gcd}, $M$ is odd, and hence $N$ is odd. Reducing the torsor modulo
$8$, we obtain
\[
1\equiv N^{2}\equiv pqM^{4}\equiv pq\pmod 8.
\]
\end{proof}
\begin{rem}
    \label{rem:3.2} 
By Lemmas~\ref{lem:oT1} and \ref{lem:oT4}, the six possible classes occur in three complementary pairs. So, for elements of $\overline{\alpha} \left( \overline{E_{-2pq}}(\bbQ) \right)$, it is enough to check the solution of the torsor $\overline{\cT_{2}} $, $\overline{\cT_{3}}$ and $\overline{\cT_{6}}$. 
\end{rem}

So, using Lemmas \ref{lem:oT2}, \ref{lem:oT3}, and \ref{lem:oT6}, we obtain the following results.
\subsection*{Case-I:} At first, we shall consider the case when $\legendre{2}{p}=1=\legendre{2}{q}=1$. It gives $(p,q)\equiv (1,1),(1,7),(7,1),(7,7)\pmod 8$.

\begin{lem}
\label{lem:oalpha_1.1} 
If $(p,q)\equiv (1,1)\pmod 8$, then
\[
\overline{\alpha}\left(\overline{E_{-2pq}}(\bbQ)\right)\subseteq
\begin{cases}
\{1,2,p,q,2p,2q,pq,2pq\},
& \text{for } \legendre{p}{q}=1,\\
\{1,2,pq,2pq\},
& \text{for } \legendre{p}{q}=-1.
\end{cases}
\]
\end{lem}

\begin{proof}
Since $p\equiv q\equiv1\pmod8$, the conditions in Lemmas~\ref{lem:oT2}
and~\ref{lem:oT3} are equivalent to $\legendre{p}{q}=1$, by quadratic
reciprocity. Moreover, all the conditions in Lemma~\ref{lem:oT6} are satisfied.
Thus, for $\legendre{p}{q}=1$, the three complementary pairs may occur, whereas
for $\legendre{p}{q}=-1$, only the pair $\{2,pq\}$ may occur. This gives the
stated inclusions.
\end{proof}

\begin{lem}
\label{lem:oalpha_1.2} 
If $(p,q)\equiv (1,7)\pmod 8$, then
\[
\overline{\alpha}\left(\overline{E_{-2pq}}(\bbQ)\right)
\begin{aligned}
    &\subseteq \{1,p,2q,2pq\},
&& \text{for } \legendre{p}{q}=1,\\
&= \{1,2pq\},
&& \text{for } \legendre{p}{q}=-1.
\end{aligned}
\]
\end{lem}

\begin{proof}
Here Lemma~\ref{lem:oT2} may contribute the pair $\{p,2q\}$ precisely when
$\legendre{p}{q}=1$. Lemma~\ref{lem:oT3} does not apply because
$q\not\equiv1\pmod8$, and Lemma~\ref{lem:oT6} does not apply because
$pq\not\equiv1\pmod8$. Hence the result follows.
\end{proof}

\begin{lem}
\label{lem:oalpha_1.3} 
If $(p,q)\equiv (7,1)\pmod 8$, then
\[
\overline{\alpha}\left(\overline{E_{-2pq}}(\bbQ)\right)
\begin{aligned}
&\subseteq \{1,q,2p,2pq\},
&& \text{for } \legendre{p}{q}=1,\\
&= \{1,2pq\},
&& \text{for } \legendre{p}{q}=-1.
\end{aligned}
\]
\end{lem}

\begin{proof}
This follows from Lemma~\ref{lem:oalpha_1.2} by interchanging $p$ and $q$.
\end{proof}

\begin{lem}
\label{lem:oalpha_1.4} 
If $(p,q)\equiv (7,7)\pmod 8$, then
\[
\overline{\alpha}\left(\overline{E_{-2pq}}(\bbQ)\right)
\subseteq \{1,2,pq,2pq\}
\]
for both $\legendre{p}{q}=1$ and $\legendre{p}{q}=-1$.
\end{lem}

\begin{proof}
Lemmas~\ref{lem:oT2} and~\ref{lem:oT3} do not apply because
$p\not\equiv1\pmod8$ and $q\not\equiv1\pmod8$, respectively. On the other hand,
all the conditions in Lemma~\ref{lem:oT6} are satisfied. Hence only the pair
$\{2,pq\}$ may occur, independently of $\legendre{p}{q}$.
\end{proof}

\subsection*{Case-II:} Now we shall consider the case when $\legendre{2}{p}=1$ and $\legendre{2}{q}=-1$. It gives $(p,q)\equiv (1,3),(1,5),(7,3),(7,5)\pmod 8$.

\begin{lem}
\label{lem:oalpha_2.1}
If $(p,q)\equiv (1,3)\pmod 8$, then
\[
\overline{\alpha}\left(\overline{E_{-2pq}}(\bbQ)\right)
\begin{aligned}
&\subseteq \{1,p,2q,2pq\},
&& \text{for } \legendre{p}{q}=1,\\
&= \{1,2pq\},
&& \text{for } \legendre{p}{q}=-1.
\end{aligned}
\]
\end{lem}

\begin{proof}
Since $p\equiv1\pmod8$, Lemma~\ref{lem:oT2} may contribute the pair
$\{p,2q\}$, and its Legendre-symbol conditions are equivalent to
$\legendre{p}{q}=1$. Lemma~\ref{lem:oT3} does not apply because
$q\not\equiv1\pmod8$, while Lemma~\ref{lem:oT6} does not apply because
$\legendre{2}{q}=-1$. Hence the result follows.
\end{proof}

\begin{lem}
\label{lem:oalpha_2.2}
If $(p,q)\equiv (1,5)\pmod 8$, then
\[
\overline{\alpha}\left(\overline{E_{-2pq}}(\bbQ)\right)
\begin{aligned}
&\subseteq \{1,p,2q,2pq\},
&& \text{for } \legendre{p}{q}=1,\\
&= \{1,2pq\},
&& \text{for } \legendre{p}{q}=-1.
\end{aligned}
\]
\end{lem}

\begin{proof}
The proof is identical to that of Lemma~\ref{lem:oalpha_2.1}.
\end{proof}


\begin{lem}
\label{lem:oalpha_2.3}
If $(p,q)\equiv (7,3)\pmod 8$, then
\[
\overline{\alpha}\left(\overline{E_{-2pq}}(\bbQ)\right)
= \{1,2pq\} \quad \text{ for both } \legendre{p}{q}=1 \quad \text{and } \legendre{p}{q}=-1.
\]
\end{lem}

\begin{proof}
Lemma~\ref{lem:oT2} does not apply because $p\not\equiv1\pmod8$,
Lemma~\ref{lem:oT3} does not apply because $q\not\equiv1\pmod8$, and
Lemma~\ref{lem:oT6} does not apply because $\legendre{2}{q}=-1$.
Thus no additional complementary pair can occur.
\end{proof}


\begin{lem}
\label{lem:oalpha_2.4}
If $(p,q)\equiv (7,5)\pmod 8$, then
\[
\overline{\alpha}\left(\overline{E_{-2pq}}(\bbQ)\right)
= \{1,2pq\} \quad \text{ for both } \legendre{p}{q}=1 \quad \text{and } \legendre{p}{q}=-1.
\]
\end{lem}

\begin{proof}
The proof is identical to that of Lemma~\ref{lem:oalpha_2.3}.
\end{proof}

\subsection*{Case-III:} Now we consider the case when $\legendre{2}{p}=-1$ and $\legendre{2}{q}=1$. It gives $(p,q)\equiv (3,1),(3,7),(5,1),(5,7)\pmod 8$.

These four cases are obtained from Case-II by interchanging $p$ and $q$ that is following Lemmas~\ref{lem:oalpha_2.1}-\ref{lem:oalpha_2.4}.

\begin{lem}
\label{lem:oalpha_3.1}
If $(p,q)\equiv (3,1)\pmod 8$, then
\[
\overline{\alpha}\left(\overline{E_{-2pq}}(\bbQ)\right)
\begin{aligned}
&\subseteq \{1,q,2p,2pq\},
&& \text{for } \legendre{p}{q}=1,\\
&= \{1,2pq\},
&& \text{for } \legendre{p}{q}=-1.
\end{aligned}
\]
\end{lem}


\begin{lem}
\label{lem:oalpha_3.2}
If $(p,q)\equiv (3,7)\pmod 8$, then
\[
\overline{\alpha}\left(\overline{E_{-2pq}}(\bbQ)\right)
= \{1,2pq\}\quad \text{ for both } \legendre{p}{q}=1 \quad \text{and } \legendre{p}{q}=-1.
\]
\end{lem}


\begin{lem}
\label{lem:oalpha_3.3}
If $(p,q)\equiv (5,1)\pmod 8$, then
\[
\overline{\alpha}\left(\overline{E_{-2pq}}(\bbQ)\right)
\begin{aligned}
&\subseteq \{1,q,2p,2pq\},
&& \text{for } \legendre{p}{q}=1,\\
&= \{1,2pq\},
&& \text{for } \legendre{p}{q}=-1.
\end{aligned}
\]
\end{lem}


\begin{lem}
\label{lem:oalpha_3.4}
If $(p,q)\equiv (5,7)\pmod 8$, then
\[
\overline{\alpha}\left(\overline{E_{-2pq}}(\bbQ)\right)
= \{1,2pq\} \quad \text{ for both } \legendre{p}{q}=1 \quad \text{and } \legendre{p}{q}=-1.
\]
\end{lem}


\subsection*{Case IV:}


Finally, we shall consider the case when $\legendre{2}{p}=-1$ and $\legendre{2}{q}=-1$. It gives $(p,q)\equiv (3,3),(3,5),(5,3),(5,5)\pmod 8$.

In all four residue classes, neither $p$ nor $q$ is congruent to $1\pmod8$;
therefore Lemmas~\ref{lem:oT2} and~\ref{lem:oT3} do not apply. Moreover,
$\legendre{2}{p}=\legendre{2}{q}=-1$, so Lemma~\ref{lem:oT6} does not apply.
Thus the same argument proves all four lemmas below.

\begin{lem}
\label{lem:oalpha_4.1} 
If $(p,q)\equiv (3,3)\pmod 8$, then
\[
\overline{\alpha}\left(\overline{E_{-2pq}}(\bbQ)\right)
= \{1,2pq\} \quad \text{ for both } \legendre{p}{q}=1 \quad \text{and } \legendre{p}{q}=-1.
\]
\end{lem}

\begin{proof}
The preceding observation rules out all three representative torsors
$\overline{\cT_2}$, $\overline{\cT_3}$ and $\overline{\cT_6}$.
Hence the image is $\{1,2pq\}$.
\end{proof}

Following the above proof of Lemma~\ref{lem:oalpha_4.1}, one can easily obtain the next three cases.
\begin{lem}
\label{lem:oalpha_4.2} 
If $(p,q)\equiv (3,5)\pmod 8$, then
\[
\overline{\alpha}\left(\overline{E_{-2pq}}(\bbQ)\right)
= \{1,2pq\} \quad \text{ for both } \legendre{p}{q}=1 \quad \text{and } \legendre{p}{q}=-1.
\]
\end{lem}


\begin{lem}
\label{lem:oalpha_4.3} 
If $(p,q)\equiv (5,3)\pmod 8$, then
\[
\overline{\alpha}\left(\overline{E_{-2pq}}(\bbQ)\right)
= \{1,2pq\} \quad \text{ for both } \legendre{p}{q}=1 \quad \text{and } \legendre{p}{q}=-1.
\]
\end{lem}


\begin{lem}
\label{lem:oalpha_4.4}
If $(p,q)\equiv (5,5)\pmod 8$, then
\[
\overline{\alpha}\left(\overline{E_{-2pq}}(\bbQ)\right)
= \{1,2pq\} \quad \text{ for both } \legendre{p}{q}=1 \quad \text{and } \legendre{p}{q}=-1.
\]
\end{lem}


\section{Proof of Theorems \ref{thm:1.1}-\ref{thm:1.4}}

Now using \S 2 and  \S 3, we prove Theorems \ref{thm:1.1}-\ref{thm:1.4}.

\subsection*{Proof of Theorem \ref{thm:1.1}}
Let $r=\rank(E_{-2pq}(\bbQ))$. By Lemmas~\ref{lem:alpha_case_1.1}
and~\ref{lem:oalpha_1.1},
\[
\begin{aligned}
|\alpha(E_{-2pq}(\bbQ))|&\leq16,
&\left|\overline{\alpha}\left(\overline{E_{-2pq}}(\bbQ)\right)\right|&\leq8,
&&\text{for }\legendre{p}{q}=1,\\
|\alpha(E_{-2pq}(\bbQ))|&\leq8,
&\left|\overline{\alpha}\left(\overline{E_{-2pq}}(\bbQ)\right)\right|&\leq4,
&&\text{for }\legendre{p}{q}=-1.
\end{aligned}
\]
Therefore, Proposition~\ref{prop:2.1} gives
$ \displaystyle
r\leq5 \quad\text{for }\legendre{p}{q}=1,$ and $ \displaystyle r\leq3 \quad\text{for }\legendre{p}{q}=-1.
$
By Lemma~\ref{lem:root}, $w(E_{-2pq}/\bbQ)=-1$. Hence, assuming the
Parity Conjecture, $r$ is odd. Consequently,
\[
\rank(E_{-2pq}(\bbQ))=
\begin{cases}
1,3,\text{ or }5, & \text{for }\legendre{p}{q}=1,\\
1\text{ or }3, & \text{for }\legendre{p}{q}=-1.
\end{cases}
\]

\subsection*{Proof of Theorem \ref{thm:1.2}}

\begin{proof}
For the six residue classes occurring in the theorem, the relevant pairs of
lemmas are
\[
\begin{array}{c|c}
(p,q)\pmod8 & \text{lemmas used}\\
\hline
(1,3) & \ref{lem:alpha_case_2.1},\ \ref{lem:oalpha_2.1}\\
(3,1) & \ref{lem:alpha_case_3.1},\ \ref{lem:oalpha_3.1}\\
(1,5) & \ref{lem:alpha_case_2.2},\ \ref{lem:oalpha_2.2}\\
(5,1) & \ref{lem:alpha_case_3.3},\ \ref{lem:oalpha_3.3}\\
(1,7) & \ref{lem:alpha_case_1.2},\ \ref{lem:oalpha_1.2}\\
(7,1) & \ref{lem:alpha_case_1.3},\ \ref{lem:oalpha_1.3}
\end{array}
\]
In every row, the corresponding lemmas give
\[
\begin{aligned}
|\alpha(E_{-2pq}(\bbQ))|&\leq8,
&\left|\overline{\alpha}\left(\overline{E_{-2pq}}(\bbQ)\right)\right|&\leq4,
&&\text{for }\legendre{p}{q}=1,\\
|\alpha(E_{-2pq}(\bbQ))|&\leq4,
&\left|\overline{\alpha}\left(\overline{E_{-2pq}}(\bbQ)\right)\right|&=2,
&&\text{for }\legendre{p}{q}=-1.
\end{aligned}
\]
Thus Proposition~\ref{prop:2.1} yields
$\displaystyle
r\leq3 \quad\text{for }\legendre{p}{q}=1,
\qquad
r\leq1 \quad\text{for }\legendre{p}{q}=-1.
$
By Lemma~\ref{lem:root} and the Parity Conjecture, $r$ is odd. Hence
\[
\rank(E_{-2pq}(\bbQ))=
\begin{cases}
1\text{ or }3, & \text{for }\legendre{p}{q}=1,\\
1, & \text{for }\legendre{p}{q}=-1.
\end{cases}
\]
\end{proof}

\subsection*{Proof of Theorem \ref{thm:1.3}}

\begin{proof}
For the eight residue classes in the theorem, use the following pairs of lemmas:
\[
\begin{array}{c|c}
(p,q)\pmod8 & \text{lemmas used}\\
\hline
(3,3) & \ref{lem:alpha_case_4.1},\ \ref{lem:oalpha_4.1}\\
(3,5) & \ref{lem:alpha_case_4.2},\ \ref{lem:oalpha_4.2}\\
(5,3) & \ref{lem:alpha_case_4.3},\ \ref{lem:oalpha_4.3}\\
(5,5) & \ref{lem:alpha_case_4.4},\ \ref{lem:oalpha_4.4}\\
(3,7) & \ref{lem:alpha_case_3.2},\ \ref{lem:oalpha_3.2}\\
(7,3) & \ref{lem:alpha_case_2.3},\ \ref{lem:oalpha_2.3}\\
(5,7) & \ref{lem:alpha_case_3.4},\ \ref{lem:oalpha_3.4}\\
(7,5) & \ref{lem:alpha_case_2.4},\ \ref{lem:oalpha_2.4}
\end{array}
\]
For either value of $\legendre{p}{q}$, the cited lemmas give
$\displaystyle|\alpha(E_{-2pq}(\bbQ))|\leq4,$ and $
\displaystyle
\left|\overline{\alpha}\left(\overline{E_{-2pq}}(\bbQ)\right)\right|=2.
$
Hence Proposition~\ref{prop:2.1} gives $r\leq1$. On the other hand,
Lemma~\ref{lem:root} and the Parity Conjecture imply that $r$ is odd.
Therefore $r=1$ in all eight residue classes, for both
$\legendre{p}{q}=1$ and $\legendre{p}{q}=-1$.
\end{proof}

\subsection*{Proof of Theorem \ref{thm:1.4}}

\begin{proof}
Assume $(p,q)\equiv(7,7)\pmod8$. By Lemmas~\ref{lem:alpha_case_1.4}
and~\ref{lem:oalpha_1.4}, we obtain for both $\legendre{p}{q}=1$ and $\legendre{p}{q}=-1$, $ \displaystyle |\alpha(E_{-2pq}(\bbQ))|\leq8,$ and $ \displaystyle
\left|\overline{\alpha}\left(\overline{E_{-2pq}}(\bbQ)\right)\right|\leq4.
$
Therefore Proposition~\ref{prop:2.1} gives $r\leq3$. Since
$w(E_{-2pq}/\bbQ)=-1$ by Lemma~\ref{lem:root}, the Parity Conjecture implies
that $r$ is odd. Thus $\rank(E_{-2pq}(\bbQ))=1\text{ or }3$.

\end{proof}


\section{Proof of Theorem \ref{thm:1.5}}
In each case, we exhibit an integral solution of the relevant torsor.
The corresponding lemmas determine the images of $\alpha$ and
$\overline{\alpha}$, and Proposition~\ref{prop:2.1} then gives
$\operatorname{rank}(E_{-2pq}(\bbQ))=1$.

\smallskip
\noindent\textbf{(i)}
For $(p,q)\equiv(3,3)\pmod 8$, substituting $M=2$ and $e=1$ into
$\cT_3$ gives
\[
N^2=-2M^4+pqe^4=pq-32.
\]
Thus, $\cT_3$ has the integral solution
$(N,M,e)=(\sqrt{pq-32},2,1)$. By
Lemmas~\ref{lem:alpha_case_4.1} and \ref{lem:oalpha_4.1},
\[
\alpha(E_{-2pq}(\bbQ))=\{1,pq,-2,-2pq\},
\qquad
\overline{\alpha}(\overline{E_{-2pq}}(\bbQ))
=\{1,2pq\}.
\]
Hence, the cardinalities of these two sets are $4$ and $2$,
respectively. Proposition~\ref{prop:2.1} gives the result for both
$\legendre{p}{q}=1$ and $\legendre{p}{q}=-1$.

\begin{table}[!ht]
\centering
\scriptsize
\renewcommand{\arraystretch}{0.9}
\setlength{\tabcolsep}{5pt}
\begin{tabular}{cccc}
\hline
$p$ & $q$ & $N=\sqrt{pq-32}$ &
$\operatorname{rank}(E_{-2pq}(\bbQ))$\\
\hline
$11$ & $3$  & $1$  & $1$\\
$19$ & $3$  & $5$  & $1$\\
$43$ & $11$ & $21$ & $1$\\
\hline
\end{tabular}
\caption{Examples corresponding to part~\textup{(i)}.}
\label{tab:4}
\end{table}

\smallskip
\noindent\textbf{(ii)}
Suppose that $(p,q)\equiv(3,5)\pmod 8$.
For $\legendre{p}{q}=1$, putting $M=e=1$ in $\cT_6$ gives
$N^2=p-2q$; hence
$(N,M,e)=(\sqrt{p-2q},1,1)$ is an integral solution.
For $\legendre{p}{q}=-1$, putting $M=e=1$ in $\cT_5$ gives
$N^2=2p-q$; hence
$(N,M,e)=(\sqrt{2p-q},1,1)$ is an integral solution.
The result now follows from
Lemmas~\ref{lem:alpha_case_4.2} and \ref{lem:oalpha_4.2},
together with Proposition~\ref{prop:2.1}.

\begin{table}[!ht]
\centering
\scriptsize
\renewcommand{\arraystretch}{0.9}
\setlength{\tabcolsep}{2pt}
\begin{tabular*}{\textwidth}{@{\extracolsep{\fill}}cccc|cccc@{}}
\hline
\multicolumn{4}{c}{$\legendre{p}{q}=1$}&
\multicolumn{4}{c}{$\legendre{p}{q}=-1$}\\
\hline
$p$&$q$&$N=\sqrt{p-2q}$&Rank&
$p$&$q$&$N=\sqrt{2p-q}$&Rank\\
\hline
$11$&$5$&$1$&$1$  &$11$&$13$&$3$&$1$\\
$19$&$5$&$3$&$1$  &$19$&$13$&$5$&$1$\\
$59$&$5$&$7$&$1$  &$43$&$5$ &$9$&$1$\\
\hline
\end{tabular*}
\caption{Examples corresponding to part~\textup{(ii)}.}
\label{tab:5}
\end{table}

\smallskip
\noindent\textbf{(iii)}
Suppose that $(p,q)\equiv(5,3)\pmod 8$.
For $\legendre{p}{q}=1$, putting $M=e=1$ in $\cT_4$ gives
$N^2=q-2p$, while for $\legendre{p}{q}=-1$, putting $M=e=1$
in $\cT_7$ gives $N^2=2q-p$. Thus, the relevant integral solutions are
\[
(\sqrt{q-2p},1,1)
\quad\text{and}\quad
(\sqrt{2q-p},1,1),
\]
respectively. Equivalently, both conclusions follow from part~\textup{(ii)}
by interchanging $p$ and $q$.

\begin{table}[!ht]
\centering
\scriptsize
\renewcommand{\arraystretch}{0.9}
\setlength{\tabcolsep}{2pt}
\begin{tabular*}{\textwidth}{@{\extracolsep{\fill}}cccc|cccc@{}}
\hline
\multicolumn{4}{c}{$\legendre{p}{q}=1$}&
\multicolumn{4}{c}{$\legendre{p}{q}=-1$}\\
\hline
$p$&$q$&$N=\sqrt{q-2p}$&Rank&
$p$&$q$&$N=\sqrt{2q-p}$&Rank\\
\hline
$5$&$11$&$1$&$1$ &$13$&$11$&$3$&$1$\\
$5$&$19$&$3$&$1$ &$13$&$19$&$5$&$1$\\
$5$&$59$&$7$&$1$ &$5$ &$43$&$9$&$1$\\
\hline
\end{tabular*}
\caption{Examples corresponding to part~\textup{(iii)}.}
\label{tab:6}
\end{table}

\smallskip
\noindent\textbf{(iv)}
Suppose that $(p,q)\equiv(3,7)\pmod 8$.
For $\legendre{p}{q}=1$, putting $M=1$ and $e=2$ in $\cT_5$
gives
\[
N^2=-qM^4+2pe^4=32p-q,
\]
and hence the integral solution
$(N,M,e)=(\sqrt{32p-q},1,2)$.
For $\legendre{p}{q}=-1$, putting $M=e=1$ in $\cT_4$ gives
$N^2=q-2p$ and the integral solution
$(N,M,e)=(\sqrt{q-2p},1,1)$.
The corresponding descent lemmas and Proposition~\ref{prop:2.1}
give the result.

\begin{table}[!ht]
\centering
\scriptsize
\renewcommand{\arraystretch}{0.9}
\setlength{\tabcolsep}{2pt}
\begin{tabular*}{\textwidth}{@{\extracolsep{\fill}}cccc|cccc@{}}
\hline
\multicolumn{4}{c}{$\legendre{p}{q}=1$}&
\multicolumn{4}{c}{$\legendre{p}{q}=-1$}\\
\hline
$p$&$q$&$N=\sqrt{32p-q}$&Rank&
$p$&$q$&$N=\sqrt{q-2p}$&Rank\\
\hline
$3$ &$71$ &$5$&$1$ &$3$ &$7$ &$1$&$1$\\
$11$&$271$&$9$&$1$ &$11$&$23$&$1$&$1$\\
$19$&$599$&$3$&$1$ &$11$&$71$&$7$&$1$\\
\hline
\end{tabular*}
\caption{Examples corresponding to part~\textup{(iv)}.}
\label{tab:7}
\end{table}

\smallskip
\noindent\textbf{(v)}
Suppose that $(p,q)\equiv(7,3)\pmod 8$.
For $\legendre{p}{q}=1$, putting $M=e=1$ in $\cT_6$ gives
$N^2=p-2q$ and the solution
$(N,M,e)=(\sqrt{p-2q},1,1)$.
For $\legendre{p}{q}=-1$, putting $M=1$ and $e=2$ in $\cT_7$
gives $N^2=32q-p$ and the solution
$(N,M,e)=(\sqrt{32q-p},1,2)$.
Equivalently, this follows from part~\textup{(iv)} by interchanging
$p$ and $q$.

\begin{table}[!ht]
\centering
\scriptsize
\renewcommand{\arraystretch}{0.9}
\setlength{\tabcolsep}{2pt}
\begin{tabular*}{\textwidth}{@{\extracolsep{\fill}}cccc|cccc@{}}
\hline
\multicolumn{4}{c}{$\legendre{p}{q}=1$}&
\multicolumn{4}{c}{$\legendre{p}{q}=-1$}\\
\hline
$p$&$q$&$N=\sqrt{p-2q}$&Rank&
$q$&$p$&$N=\sqrt{32q-p}$&Rank\\
\hline
$7$ &$3$ &$1$&$1$ &$3$ &$71$ &$5$&$1$\\
$23$&$11$&$1$&$1$ &$11$&$271$&$9$&$1$\\
$71$&$11$&$7$&$1$ &$19$&$599$&$3$&$1$\\
\hline
\end{tabular*}
\caption{Examples corresponding to part~\textup{(v)}.}
\label{tab:8}
\end{table}

\smallskip
\noindent\textbf{(vi)}
Suppose that $(p,q)\equiv(5,7)\pmod 8$.
For $\legendre{p}{q}=1$, putting $M=1$ and $e=2$ in $\cT_5$
gives $N^2=32p-q$ and the solution
$(N,M,e)=(\sqrt{32p-q},1,2)$.
For $\legendre{p}{q}=-1$, putting $M=e=1$ in $\cT_7$ gives
$N^2=2q-p$ and the solution
$(N,M,e)=(\sqrt{2q-p},1,1)$.
The corresponding descent lemmas and Proposition~\ref{prop:2.1}
give the result.

\begin{table}[!ht]
\centering
\scriptsize
\renewcommand{\arraystretch}{0.9}
\setlength{\tabcolsep}{2pt}
\begin{tabular*}{\textwidth}{@{\extracolsep{\fill}}cccc|cccc@{}}
\hline
\multicolumn{4}{c}{$\legendre{p}{q}=1$}&
\multicolumn{4}{c}{$\legendre{p}{q}=-1$}\\
\hline
$p$&$q$&$N=\sqrt{32p-q}$&Rank&
$p$&$q$&$N=\sqrt{2q-p}$&Rank\\
\hline
$5$ &$79$ &$9$ &$1$ &$5$ &$7$ &$3$&$1$\\
$13$&$191$&$15$&$1$ &$13$&$31$&$7$&$1$\\
$13$&$367$&$7$ &$1$ &$61$&$71$&$9$&$1$\\
\hline
\end{tabular*}
\caption{Examples corresponding to part~\textup{(vi)}.}
\label{tab:9}
\end{table}

\smallskip
\noindent\textbf{(vii)}
Suppose that $(p,q)\equiv(7,5)\pmod 8$.
For $\legendre{p}{q}=1$, putting $M=e=1$ in $\cT_5$ gives
$N^2=2p-q$ and the solution
$(N,M,e)=(\sqrt{2p-q},1,1)$.
For $\legendre{p}{q}=-1$, putting $M=1$ and $e=2$ in $\cT_7$
gives $N^2=32q-p$ and the solution
$(N,M,e)=(\sqrt{32q-p},1,2)$.
Equivalently, the result follows from part~\textup{(vi)} by
interchanging $p$ and $q$.


\begin{table}[!ht]
\centering
\scriptsize
\renewcommand{\arraystretch}{0.9}
\setlength{\tabcolsep}{2pt}
\begin{tabular*}{\textwidth}{@{\extracolsep{\fill}}cccc|cccc@{}}
\hline
\multicolumn{4}{c}{$\legendre{p}{q}=1$}&
\multicolumn{4}{c}{$\legendre{p}{q}=-1$}\\
\hline
$p$&$q$&$N=\sqrt{2p-q}$&Rank&
$p$&$q$&$N=\sqrt{32q-p}$&Rank\\
\hline
$79$ &$5$ &$9$ &$1$& $7$ &$5$ &$3$&$1$ \\
$191$&$13$&$15$&$1$&$31$&$13$&$7$&$1$ \\
$367$&$13$&$7$ &$1$ & $71$&$61$&$9$&$1$\\
\hline
\end{tabular*}
\caption{Examples corresponding to part~\textup{(vii)}.}
\label{tab:10}
\end{table}

\smallskip
\noindent\textbf{(viii)}
Suppose that $(p,q)\equiv(5,5)\pmod 8$ and that $2pq-1$ is a
perfect square. Putting $M=e=1$ in $\cT_1$ gives
\[
N^2=-M^4+2pqe^4=2pq-1.
\]
Thus, $\cT_1$ has the integral solution
$(N,M,e)=(\sqrt{2pq-1},1,1)$.
The corresponding descent lemmas and Proposition~\ref{prop:2.1}
give the result for both $\legendre{p}{q}=1$ and
$\legendre{p}{q}=-1$.

\begin{table}[!ht]
\centering
\scriptsize
\renewcommand{\arraystretch}{0.9}
\setlength{\tabcolsep}{5pt}
\begin{tabular}{cccc}
\hline
$p$&$q$&$N=\sqrt{2pq-1}$&
$\operatorname{rank}(E_{-2pq}(\bbQ))$\\
\hline
$5$ &$29$&$17$&$1$\\
$13$&$37$&$31$&$1$\\
\hline
\end{tabular}
\caption{Examples corresponding to part~\textup{(viii)}.}
\label{tab:11}
\end{table}
\qed

\end{document}